\documentclass[11pt,a4paper]{amsart}
\usepackage{amsfonts,amsmath,amssymb,amsthm,mathtools}
\usepackage[noadjust]{cite}
\usepackage[]{hyperref}

\numberwithin{equation}{section}

\DeclareMathOperator{\Span}{span}
\DeclareMathOperator{\Dom}{Dom}

\DeclarePairedDelimiter{\abs}{|}{|}

\newcommand{\dd}{\mathrm d}

\newcommand{\per}{\mathrm{per}}

\newcommand{\CC}{\mathbb{C}}
\newcommand{\NN}{\mathbb{N}}
\newcommand{\RR}{\mathbb{R}}

\newcommand{\cD}{{\mathcal D}}
\newcommand{\cH}{{\mathcal H}}
\newcommand{\cK}{{\mathcal K}}

\newcommand{\ft}{{\mathfrak t}}

\theoremstyle{plain}
\newtheorem{theorem}{Theorem}[section]
\newtheorem{proposition}[theorem]{Proposition}
\newtheorem{lemma}[theorem]{Lemma}

\theoremstyle{definition}
\newtheorem{example}[theorem]{Example}

\theoremstyle{remark}

\title[The Laplacian on Cartesian products]{The Laplacian on Cartesian products with mixed boundary conditions}
\date{}

\author[A.\ Seelmann]{Albrecht Seelmann}
\address{A.~Seelmann,
 Technische Univer\-si\-t\"at Dortmund, Fakult\"at f\"ur Mathematik, D-44221 Dortmund, Germany}
\email{albrecht.seelmann@math.tu-dortmund.de}

\begin{document}

%%%%%%%%%%%%%%%%%%%%%%%%%%%%%%%%%%%%%%%%%%%%%%%%%%%%%%%%%%%%%%%%%%%%%%%%%%%%%%%%%%%%%%%%%%%%%%%%%%%%%%%%%%%%%%%%%%%%%%%%%%%%%%%%%
%%% Abstract
%%%%%%%%%%%%%%%%%%%%%%%%%%%%%%%%%%%%%%%%%%%%%%%%%%%%%%%%%%%%%%%%%%%%%%%%%%%%%%%%%%%%%%%%%%%%%%%%%%%%%%%%%%%%%%%%%%%%%%%%%%%%%%%%%
\begin{abstract}
	A definition of the Laplacian on Cartesian products with mixed boundary conditions using quadratic forms is proposed. Its
	consistency with the standard definition for homogeneous and certain mixed boundary conditions is proved and, as a consequence,
	tensor representations of the corresponding Sobolev spaces of first order are derived. Moreover, a criterion for the domain to
	belong to the Sobolev space of second order is proved.
\end{abstract}

\maketitle

%%%%%%%%%%%%%%%%%%%%%%%%%%%%%%%%%%%%%%%%%%%%%%%%%%%%%%%%%%%%%%%%%%%%%%%%%%%%%%%%%%%%%%%%%%%%%%%%%%%%%%%%%%%%%%%%%%%%%%%%%%%%%%%%%
%%%%%%%%%%%%%%%%%%%%%%%%%%%%%%%%%%%%%%%%%%%%%%%%%%%%%%%%%%%%%%%%%%%%%%%%%%%%%%%%%%%%%%%%%%%%%%%%%%%%%%%%%%%%%%%%%%%%%%%%%%%%%%%%%
%%% Introduction
%%%%%%%%%%%%%%%%%%%%%%%%%%%%%%%%%%%%%%%%%%%%%%%%%%%%%%%%%%%%%%%%%%%%%%%%%%%%%%%%%%%%%%%%%%%%%%%%%%%%%%%%%%%%%%%%%%%%%%%%%%%%%%%%%
%%%%%%%%%%%%%%%%%%%%%%%%%%%%%%%%%%%%%%%%%%%%%%%%%%%%%%%%%%%%%%%%%%%%%%%%%%%%%%%%%%%%%%%%%%%%%%%%%%%%%%%%%%%%%%%%%%%%%%%%%%%%%%%%%
\section{Introduction}

For $j = 1, \dots, n$, $n \geq 2$, let $\Omega_j \subset \RR^{d_j}$ with $d_j \in \NN$ be open, and consider
$\Omega = \bigtimes_{j=1}^n \Omega_j \subset \RR^d$ with $d = d_1 + \dots + d_n$. For each $j$ let $\cH_j$ be a closed subspace
of $H^1(\Omega_j)$ that is dense in $L^2(\Omega_j)$; possible choices for $\cH_j$ are, for instance, $H^1(\Omega_j)$,
$H_0^1(\Omega_j)$, $H_\per^1(\Omega_j)$ (if applicable), and the subspace of $H^1(\Omega_j)$ corresponding to quasiperiodic
boundary conditions (if applicable).

For $f_j \colon \Omega_j \to \CC$, $f_j \in \cH_j$, the elementary tensor $f := \bigotimes_{j=1}^n f_j \colon \Omega \to \CC$
denotes the function in $H^1(\Omega)$ given by $f(x_1,\dots,x_n)=\prod_{j=1}^n f_j(x_j)$ with $x_j \in \Omega_j$. Set
\[
	\cK := \Span_\CC \biggl\{ \bigotimes_{j=1}^n f_j \colon f_j \in \cH_j \biggr\} \subset H^1(\Omega)
\]
and
\[
	\cH := \bigotimes_{j=1}^n \cH_j := \overline{\cK}^{H^1(\Omega)}
	\subset
	H^1(\Omega)
	,
\]
where the bar denotes the closure in $H^1(\Omega)$. For a more detailed discussion of tensor products of Hilbert spaces, the
reader is referred to~\cite[Section~7.5]{Schm12},~\cite[Section~3.4]{Wei80}, and~\cite[Section~II.4]{RS80}.

Consider on $L^2(\Omega)$ the form $\ft$ with $\Dom[\ft] = \cH$ and
\[
	\ft[ f , g ]
	=
	\sum_{k=1}^d \langle \partial_k f , \partial_k g \rangle_{L^2(\Omega)}
	,\quad
	f,g \in \cH
	.
\]
This form is closed, nonnegative, and densely defined. Thus, there exists a unique (nonnegative) self-adjoint operator $T$ on
$L^2(\Omega)$ with $\Dom(T) \subset \cH$ and
\[
	\ft[ f , g ]
	=
	\langle Tf , g \rangle_{L^2(\Omega)}
\]
for all $f \in \Dom(T)$ and $g \in \cH$.

The purpose of the present short note is to propose the above definition of $T$ as the definition of the negative Laplacian
$-\Delta$ on $L^2(\Omega)$ with mixed boundary conditions corresponding to the ones encoded in the $\cH_j$. To this end, we show
that this way to define $T$ is consistent with the definition of the Laplacian on $\Omega$ with homogeneous Dirichlet, Neumann,
and \linebreak(quasi-)periodic and also certain mixed boundary conditions. As a byproduct, this also establishes the tensor
representations $H_0^1(\Omega) = \bigotimes_{j=1}^n H_0^1(\Omega_j)$, $H^1(\Omega) = \bigotimes_{j=1}^n H^1(\Omega_j)$, and
$H_\per^1(\Omega) = \bigotimes_{j=1}^n H_\per^1(\Omega_j)$. The latter seem to be folklore, but the present author is not aware
of any suitable reference in the literature for the general setting above; the case $\Omega = \RR^d$ is clear, whereas the recent
work~\cite{AR16} establishes the first two representations (for $n=2$) only for certain bounded domains such as bounded Lipschitz
domains.

With the above notations, it is easy to see that
\begin{equation}\label{eq:tensorInd}
  \cH
  =
  \biggl( \bigotimes_{j=1}^{n-1} \cH_j \biggr) \otimes \cH_n
  ,
\end{equation}
so that we may restrict our considerations to the case $n = 2$ by induction. The rest of this note is now organized as follows:

In Section~\ref{sec:tensor}, we give an alternative characterization for $T$ in the case $n = 2$ using tensor products of
operators which makes the operator more accessible for the remaining considerations. Section~\ref{sec:consistency} then proves
the consistency of the above definition with the case of standard homogeneous boundary conditions, and tensor representations for
the corresponding Sobolev spaces of first order are derived. Also compatible mixed boundary conditions are considered here.
Finally, Section~\ref{sec:domain} is devoted to a criterion under which the domain of the operator $T$ belongs to $H^2(\Omega)$.
As an example, we discuss the particular case where each $\Omega_j$ is one dimensional.

%%%%%%%%%%%%%%%%%%%%%%%%%%%%%%%%%%%%%%%%%%%%%%%%%%%%%%%%%%%%%%%%%%%%%%%%%%%%%%%%%%%%%%%%%%%%%%%%%%%%%%%%%%%%%%%%%%%%%%%%%%%%%%%%%
%%%%%%%%%%%%%%%%%%%%%%%%%%%%%%%%%%%%%%%%%%%%%%%%%%%%%%%%%%%%%%%%%%%%%%%%%%%%%%%%%%%%%%%%%%%%%%%%%%%%%%%%%%%%%%%%%%%%%%%%%%%%%%%%%
%%% Alternative characterization
%%%%%%%%%%%%%%%%%%%%%%%%%%%%%%%%%%%%%%%%%%%%%%%%%%%%%%%%%%%%%%%%%%%%%%%%%%%%%%%%%%%%%%%%%%%%%%%%%%%%%%%%%%%%%%%%%%%%%%%%%%%%%%%%%
%%%%%%%%%%%%%%%%%%%%%%%%%%%%%%%%%%%%%%%%%%%%%%%%%%%%%%%%%%%%%%%%%%%%%%%%%%%%%%%%%%%%%%%%%%%%%%%%%%%%%%%%%%%%%%%%%%%%%%%%%%%%%%%%%
\section{An alternative characterization using tensor products of operators}
\label{sec:tensor}

Throughout this section we consider the case $n = 2$, so that $\Omega = \Omega_1 \times \Omega_2$ and
$\cH = \cH_1 \otimes \cH_2$.

For $j \in \{ 1,2 \}$, let $\Delta_j$ be the Laplacian on $\Omega_j$ with form domain $\cH_j$, that is, $-\Delta_j$ is the unique
nonnegative self-adjoint operator on $L^2(\Omega_j)$ satisfying $\Dom(\Delta_j) \subset \cH_j$ and
\[
	\sum_{k=1}^{d_j} \langle \partial_k u , \partial_k v \rangle_{L^2(\Omega_j)}
	=
	\langle (-\Delta_j)u , v \rangle_{L^2(\Omega_j)}
\]
for all $u \in \Dom(\Delta_j)$ and $v \in \cH_j$. Consider the operator $S$ on $L^2(\Omega)$ defined by
\[
  S
  :=
  (-\Delta_1) \otimes I_2 + I_1 \otimes (-\Delta_2)
  =:
  S_1 + S_2
\]
where $I_j$ denotes the identity operator on $L^2(\Omega_j)$. According to~\cite[Theorem~7.23 and Exercise~7.17\,a.]{Schm12},
this operator is self-adjoint and nonnegative with operator core
$\cD := \Span_\CC\{ f_1 \otimes f_2 \colon f_j \in \Dom(\Delta_j) \}$;
cf.~also~\cite[Theorem~8.33]{Wei80},~\cite[Theorem~VIII.33]{RS80} and~\cite[Proposition~A.2]{MN12}.

The above gives an alternative definition of the Laplacian on $\Omega$ with mixed boundary conditions corresponding to the
$\cH_j$:
\begin{lemma}\label{lemma:LaplacianTensor}
	For $n = 2$ we have $T = S$.
\end{lemma}

\begin{proof}
  Let $f = f_1 \otimes f_2 \in \cD$ and $g = g_1 \otimes g_2 \in \cK$. By Fubini's theorem, we then have 
	\begin{align*}
		\ft[ f , g ]
		&=
		\sum_{j=1}^2 \sum_{k=1}^{d_j} \langle \partial_k f_j , \partial_k g_j \rangle_{L^2(\Omega_j)} \prod_{l \neq j}
			\langle f_l , g_l \rangle_{L^2(\Omega_l)}\\
		&=
    \langle (-\Delta_1) f_1 , g_1 \rangle_{L^2(\Omega_1)} \langle f_2 , g_2 \rangle_{L^2(\Omega_2)}
    +
    \langle f_1 , g_1 \rangle_{L^2(\Omega_1)} \langle (-\Delta_2) f_2 , g_2 \rangle_{L^2(\Omega_2)}\\
		&=
		\langle S_1 f , g \rangle_{L^2(\Omega)} + \langle S_2 f , g \rangle_{L^2(\Omega)} = \langle S f , g \rangle_{L^2(\Omega)}
		.
	\end{align*}
	By sesquilinearity and denseness of $\cK$ in $\cH$ with respect to the $H^1$-norm, this yields
	\[
		\ft[ f , g ]
		=
		\langle Sf , g \rangle_{L^2(\Omega)}
	\]
	for all $f \in \cD \subset \cH$ and $g \in \cH$. Hence, $\cD \subset \Dom(T)$ and $Tf = Sf$ for all $f \in \cD$, that is,
	$S|_\cD \subset T$. Since $\cD$ is a core for $S$ and both $S$ and $T$ are self-adjoint, we conclude that
	$S = \overline{S|_\cD} = T$.
\end{proof}%

The above can be considered as a technical preliminary as it makes the operator $T$ more accessible for some of our purposes.

%%%%%%%%%%%%%%%%%%%%%%%%%%%%%%%%%%%%%%%%%%%%%%%%%%%%%%%%%%%%%%%%%%%%%%%%%%%%%%%%%%%%%%%%%%%%%%%%%%%%%%%%%%%%%%%%%%%%%%%%%%%%%%%%%
%%%%%%%%%%%%%%%%%%%%%%%%%%%%%%%%%%%%%%%%%%%%%%%%%%%%%%%%%%%%%%%%%%%%%%%%%%%%%%%%%%%%%%%%%%%%%%%%%%%%%%%%%%%%%%%%%%%%%%%%%%%%%%%%%
%%% Consistency with homogeneous boundary conditions
%%%%%%%%%%%%%%%%%%%%%%%%%%%%%%%%%%%%%%%%%%%%%%%%%%%%%%%%%%%%%%%%%%%%%%%%%%%%%%%%%%%%%%%%%%%%%%%%%%%%%%%%%%%%%%%%%%%%%%%%%%%%%%%%%
%%%%%%%%%%%%%%%%%%%%%%%%%%%%%%%%%%%%%%%%%%%%%%%%%%%%%%%%%%%%%%%%%%%%%%%%%%%%%%%%%%%%%%%%%%%%%%%%%%%%%%%%%%%%%%%%%%%%%%%%%%%%%%%%%
\section{Consistency with homogeneous and mixed boundary conditions of product type}
\label{sec:consistency}

In this section, we show that the definition of $T$ (with general $n \geq 2$) agrees with the usual definition of the Laplacian
in the case where homogeneous boundary conditions (Dirichlet, Neumann, (quasi-)periodic) are imposed on $\Omega$, that is, where
all $\cH_j$ are chosen of the same type. We also discuss the case of certain mixed boundary conditions compatible with the
product structure.

%%%%%%%%%%%%%%%%%%%%%%%%%%%%%%%%%%%%%%%%%%%%%%%%%%%%%%%%%%%%%%%%%%%%%%%%%%%%%%%%%%%%%%%%%%%%%%%%%%%%%%%%%%%%%%%%%%%%%%%%%%%%%%%%%
%%% Dirichlet b.c.
%%%%%%%%%%%%%%%%%%%%%%%%%%%%%%%%%%%%%%%%%%%%%%%%%%%%%%%%%%%%%%%%%%%%%%%%%%%%%%%%%%%%%%%%%%%%%%%%%%%%%%%%%%%%%%%%%%%%%%%%%%%%%%%%%
\subsection{Dirichlet boundary conditions}

We denote the Dirichlet Laplacian on $\Omega$ by $\Delta_\Omega^D$.

\begin{proposition}\label{prop:consistencyDirichlet}
	If $\cH_j = H_0^1(\Omega_j)$ for each $j$, then $T = -\Delta_\Omega^D$. In particular, we have
	$H_0^1(\Omega) = \bigotimes_{j=1}^n H_0^1(\Omega_j)$.
\end{proposition}

\begin{proof}
  In view of~\eqref{eq:tensorInd}, it suffices to consider the case $n = 2$ by induction.
  
	Let $f = f_1 \otimes f_2$ with $f_j \in \Dom(\Delta_j)$, and let $g \in C_c^\infty(\Omega)$. We then have
	$h := g(\cdot,x_2) \in C_c^\infty(\Omega_1)$ for every $x_2 \in \Omega_2$, and, with $S_1$ as in Section~\ref{sec:tensor} and
	using Fubini's theorem, we obtain
	\begin{align*}
		\langle S_1 f , g \rangle_{L^2(\Omega)}
		&=
		\int_{\Omega_2} f_2(x_2) \langle (-\Delta_1)f_1 , h \rangle_{L^2(\Omega_1)} \,\dd x_2\\
		&=
		\sum_{k=1}^{d_1} \int_{\Omega_2} f_2(x_2) \langle \partial_k f_1 , \partial_k h \rangle	_{L^2(\Omega_1)} \,\dd x_2\\
		&=
		\sum_{k=1}^{d_1} \langle \partial_k f , \partial_k g \rangle_{L^2(\Omega)}
		.
	\end{align*}
	In the same way we see that
	\[
		\langle S_2 f , g \rangle_{L^2(\Omega)}
		=
		\sum_{k=d_1+1}^{d_1+d_2} \langle \partial_k f , \partial_k g \rangle_{L^2(\Omega)}
		,
	\]
	and summing up gives
	\[
		\langle Sf , g \rangle_{L^2(\Omega)}
		=
		\sum_{k=1}^d \langle \partial_k f , \partial_k g \rangle_{L^2(\Omega)}
		.
	\]
	By sesquilinearity and denseness of $C_c^\infty(\Omega)$ in $H_0^1(\Omega)$, the latter extends to all
	$f \in \cD \subset H_0^1(\Omega_1) \otimes H_0^1(\Omega_2) \subset H_0^1(\Omega)$ and $g \in H_0^1(\Omega)$, which implies that
	$S|_\cD \subset -\Delta_\Omega^D$. Since $\cD$	is a core for $S$ and $S$ and $-\Delta_\Omega^D$ are self-adjoint, we conclude
	that $S = \overline{S|_\cD} = -\Delta_\Omega^D$. In view of Lemma~\ref{lemma:LaplacianTensor}, this completes the proof of the
	identity $T = -\Delta_\Omega^D$.
	
	Finally, the claimed tensor representation of $H_0^1(\Omega)$ follows since with $T = -\Delta_\Omega^D$ also the form domains
	of $T$ and $-\Delta_\Omega^D$ agree.
\end{proof}%

%%%%%%%%%%%%%%%%%%%%%%%%%%%%%%%%%%%%%%%%%%%%%%%%%%%%%%%%%%%%%%%%%%%%%%%%%%%%%%%%%%%%%%%%%%%%%%%%%%%%%%%%%%%%%%%%%%%%%%%%%%%%%%%%%
%%% Neumann b.c.
%%%%%%%%%%%%%%%%%%%%%%%%%%%%%%%%%%%%%%%%%%%%%%%%%%%%%%%%%%%%%%%%%%%%%%%%%%%%%%%%%%%%%%%%%%%%%%%%%%%%%%%%%%%%%%%%%%%%%%%%%%%%%%%%%
\subsection{Neumann boundary conditions}

We denote the Neumann Laplacian on $\Omega$ by $\Delta_\Omega^N$.

\begin{proposition}\label{prop:consistencyNeumann}
	If $\cH_j = H^1(\Omega_j)$ for each $j$, then $T = -\Delta_\Omega^N$. In particular, we have
	$H^1(\Omega) = \bigotimes_{j=1}^n H^1(\Omega_j)$.
\end{proposition}

\begin{proof}
  It again suffices to consider the case $n=2$.
  
	For $g \in H^1(\Omega)$, we have $g(\cdot,x_2) \in H^1(\Omega_1)$ for almost every $x_2 \in \Omega_2$ and
	$g(x_1,\cdot) \in H^1(\Omega_2)$ for almost every $x_1 \in \Omega_1$. As in the proof of
	Proposition~\ref{prop:consistencyDirichlet}, we therefore obtain
	\[
		\langle Sf , g \rangle_{L^2(\Omega)}
		=
		\sum_{k=1}^d \langle \partial_k f , \partial_k g \rangle_{L^2(\Omega)}
	\]	
	for all $f \in \cD \subset H^1(\Omega_1) \otimes H^1(\Omega_2) \subset H^1(\Omega)$ and all $g \in H^1(\Omega)$; here, we do not
	need to approximate $g$ but may directly	work with $g \in H^1(\Omega)$. The rest of the proof then works exactly as in the above
	proof.
\end{proof}%

%%%%%%%%%%%%%%%%%%%%%%%%%%%%%%%%%%%%%%%%%%%%%%%%%%%%%%%%%%%%%%%%%%%%%%%%%%%%%%%%%%%%%%%%%%%%%%%%%%%%%%%%%%%%%%%%%%%%%%%%%%%%%%%%%
%%% Periodic b.c.
%%%%%%%%%%%%%%%%%%%%%%%%%%%%%%%%%%%%%%%%%%%%%%%%%%%%%%%%%%%%%%%%%%%%%%%%%%%%%%%%%%%%%%%%%%%%%%%%%%%%%%%%%%%%%%%%%%%%%%%%%%%%%%%%%
\subsection{Periodic boundary conditions}

If $\Omega$ is a hyperrectangle, we denote the corresponding Laplacian with periodic boundary conditions by $\Delta_\Omega^\per$.

\begin{proposition}\label{prop:consistencyPeriodic}
	For each $j = 1,\dots,n$, let $\Omega_j$ be a hyperrectangle in $\RR^{d_j}$ and
	$\cH_j = H_\per^1(\Omega_j) := \overline{C_\per^\infty(\Omega_j)}^{H^1(\Omega_j)}$. Then $T = -\Delta_\Omega^\per$. In
	particular, we have $H_\per^1(\Omega) = \bigotimes_{j=1}^n H_\per^1(\Omega_j)$.
\end{proposition}

\begin{proof}
  For $n = 2$ and $g \in C_\per^\infty(\Omega)$ we have $g(\cdot,x_2) \in C_\per^\infty(\Omega_1)$ for all $x_2 \in \Omega_2$ and
  $g(x_1,\cdot) \in C_\per^\infty(\Omega_2)$ for all $x_1 \in \Omega_1$. We may now proceed in the same way as in the proof of
  Proposition~\ref{prop:consistencyDirichlet}.
\end{proof}%

Quasiperiodic boundary conditions work analogously.

%%%%%%%%%%%%%%%%%%%%%%%%%%%%%%%%%%%%%%%%%%%%%%%%%%%%%%%%%%%%%%%%%%%%%%%%%%%%%%%%%%%%%%%%%%%%%%%%%%%%%%%%%%%%%%%%%%%%%%%%%%%%%%%%%
%%% Mixed b.c.
%%%%%%%%%%%%%%%%%%%%%%%%%%%%%%%%%%%%%%%%%%%%%%%%%%%%%%%%%%%%%%%%%%%%%%%%%%%%%%%%%%%%%%%%%%%%%%%%%%%%%%%%%%%%%%%%%%%%%%%%%%%%%%%%%
\subsection{Mixed boundary conditions of product type}

For each $j$, let $\Gamma_j$ be a closed subset of $\partial\Omega_j$, and define the closed subset
$\Gamma \subset \partial\Omega$ by
\[
	\Gamma
	:=
	\bigl( \Gamma_1 \times \overline{\Omega}_2 \times \dots \times \overline{\Omega}_n \bigr) \cup \dots \cup
	\bigl( \overline{\Omega}_1 \times \dots \times \overline{\Omega}_{n-1} \times \Gamma_n \bigr)
	.
\]
Let
\[
	V
	:=
	\overline{ \{ u|_\Omega \colon u \in C_c^\infty(\RR^d\setminus\Gamma) \} }^{H^1(\Omega)}
	,
\]
and denote by $\Delta_\Omega^V$ the Laplacian on $\Omega$ with form domain $V$. Following~\cite[Section~4.1]{Ouh05}, this
construction can be interpreted as to encode Dirichlet boundary conditions on $\Gamma$ and Neumann boundary conditions on the
rest of the boundary.

\begin{proposition}\label{prop:consistencyMixed}
	If $\cH_j = \overline{ \{ u|_{\Omega_j} \colon u \in C_c^\infty(\RR^{d_j}\setminus\Gamma_j)\} }^{H^1(\Omega_j)}$ for each $j$,
	then $T = -\Delta_\Omega^V$ and $V = \bigotimes_{j=1}^n \cH_j$.
\end{proposition}

\begin{proof}
	It again suffices to consider the case $n = 2$.
	
	For $u \in C_c^\infty(\RR^d\setminus\Gamma)$, it is not hard to see that
	$u(\cdot,x_2) \in C_c^\infty(\RR^{d_1}\setminus\Gamma_1)$ for all $x_2 \in \Omega_2$ and
	$u(x_1,\cdot) \in C_c^\infty(\RR^{d_2}\setminus\Gamma_2)$ for all $x_1 \in \Omega_1$. Since also
	$(\RR^{d_1}\setminus\Gamma_1) \times (\RR^{d_2}\setminus\Gamma_2) \subset \RR^d\setminus\Gamma$ and, therefore,
	$\cH_1 \otimes \cH_2 \subset V$, we may now proceed just as before to conclude the claim.
\end{proof}%

For $n = 2$ and the choices $\Gamma_1 = \partial\Omega_1$ and $\Gamma_2 = \emptyset$, we have $\cH_1 = H_0^1(\Omega_1)$ and,
provided that $\Omega_2$ has the so-called~\emph{segment property}, $\cH_2 = H^1(\Omega_2)$, see,
e.g.,~\cite[Theorem~3.22]{AF03}; cf.~also~\cite[Corollary~9.8]{Bre11}. For this situation,
Proposition~\ref{prop:consistencyMixed} shows that the proposed definition of the operator $T$ is consistent with the established
one with this kind of mixed boundary conditions. However, if $\Omega_2$ does not have the segment property and
$\overline{ \{ u|_{\Omega_2} \colon u \in C_c^\infty(\RR^{d_2})\} }^{H^1(\Omega_2)} \subsetneq H^1(\Omega_2)$, for the choices
$\cH_1 = H_0^1(\Omega_1)$ and $\cH_2 = H^1(\Omega_2)$ the present author is not aware of any established way of defining the
operator to compare with.

%%%%%%%%%%%%%%%%%%%%%%%%%%%%%%%%%%%%%%%%%%%%%%%%%%%%%%%%%%%%%%%%%%%%%%%%%%%%%%%%%%%%%%%%%%%%%%%%%%%%%%%%%%%%%%%%%%%%%%%%%%%%%%%%%
%%%%%%%%%%%%%%%%%%%%%%%%%%%%%%%%%%%%%%%%%%%%%%%%%%%%%%%%%%%%%%%%%%%%%%%%%%%%%%%%%%%%%%%%%%%%%%%%%%%%%%%%%%%%%%%%%%%%%%%%%%%%%%%%%
%%% Domain in H^2
%%%%%%%%%%%%%%%%%%%%%%%%%%%%%%%%%%%%%%%%%%%%%%%%%%%%%%%%%%%%%%%%%%%%%%%%%%%%%%%%%%%%%%%%%%%%%%%%%%%%%%%%%%%%%%%%%%%%%%%%%%%%%%%%%
%%%%%%%%%%%%%%%%%%%%%%%%%%%%%%%%%%%%%%%%%%%%%%%%%%%%%%%%%%%%%%%%%%%%%%%%%%%%%%%%%%%%%%%%%%%%%%%%%%%%%%%%%%%%%%%%%%%%%%%%%%%%%%%%%
\section{Domain in \texorpdfstring{$H^2$}{H2} and an example}
\label{sec:domain}

In this last section, we prove a criterion in terms of the $\Delta_j$ under which the domain of $T$ belongs to $H^2(\Omega)$.

\begin{proposition}\label{prop:domain}
  Suppose that for $j = 1, \dots, n$ we have $\Dom(\Delta_j) \subset H^2(\Omega)$ with
  \begin{equation}\label{eq:domain}
    \sum_{\abs{\alpha}=2} \frac{1}{\alpha!} \langle \partial^\alpha u ,  \partial^\alpha v \rangle_{L^2(\Omega_j)}
    =
    \frac{1}{2} \langle \Delta_j u , \Delta_j v \rangle_{L^2(\Omega_j)}
  \end{equation}
  for all $u,v \in \Dom(\Delta_j)$. Then also $\Dom(T) \subset H^2(\Omega)$ with
  \[
    \sum_{\abs{\alpha}=2} \frac{1}{\alpha!} \langle \partial^\alpha f ,  \partial^\alpha g \rangle_{L^2(\Omega)}
    =
    \frac{1}{2} \langle Tg, Tg \rangle_{L^2(\Omega_j)}
  \]
  for all $f,g \in \Dom(T)$.
\end{proposition}

\begin{proof}
  By induction, it again suffices to consider the case $n=2$.
  
  Let $f = f_1 \otimes f_2, g = g_1 \otimes g_2 \in \cD \subset H^2(\Omega)$. Then, with the notations from
  Section~\ref{sec:tensor},
  \[
    \langle \Delta_1 f_1 , \Delta_1 g_1 \rangle_{L^2(\Omega_1)} \langle f_2 , g_2 \rangle_{L^2(\Omega_2)}
    =
    \langle S_1f , S_1g \rangle_{L^2(\Omega)}
    ,
  \]
  as well as
  \[
    \langle f_1 , g_1 \rangle_{L^2(\Omega_1)} \langle \Delta_2 f_2 , \Delta_2 g_2 \rangle_{L^2(\Omega_2)}
    =
    \langle S_2f , S_2g \rangle_{L^2(\Omega)}
  \]
  and
  \[
    \langle f_1 , \Delta_1 g_1 \rangle_{L^2(\Omega_1)} \langle f_2 , \Delta_2 g_2 \rangle_{L^2(\Omega_2)}
    =
    \langle S_1f , S_2g \rangle_{L^2(\Omega)}
    =
    \langle S_2f , S_1g \rangle_{L^2(\Omega)}
    .
  \]
  Writing $\alpha = (\beta,\gamma)$ for multiindices in $\NN_0^d = \NN_0^{d_1} \times \NN_0^{d_2}$, we therefore obtain from the
  hypotheses that
  \begin{align*}
    \sum_{\abs{\alpha}=2} \frac{1}{\alpha!} \langle \partial^\alpha f ,  &\partial^\alpha g \rangle_{L^2(\Omega)}\\
    &=
    \sum_{m=0}^2 \sum_{\abs{\beta}=m} \frac{1}{\beta!} \langle \partial^\beta f_1 , \partial^\beta g_1 \rangle_{L^2(\Omega_1)}
      \sum_{\abs{\gamma}=2-m} \frac{1}{\gamma!} \langle \partial^\gamma f_2 , \partial^\gamma g_2 \rangle_{L^2(\Omega_2)}\\
    &=
    \frac{1}{2} \langle S_1f , S_1g \rangle_{L^2(\Omega)} + \langle S_1f , S_2g \rangle_{L^2(\Omega)}
      + \frac{1}{2} \langle S_2f , S_2g \rangle_{L^2(\Omega)}\\
    &=
    \frac{1}{2} \langle Sf , Sg \rangle_{L^2(\Omega)} = \frac{1}{2} \langle Tf , Tg \rangle_{L^2(\Omega)}
    .
  \end{align*}
  The latter extends by sesquilinearity to all $f,g \in \cD$. In turn, since $\cD$ is an operator core for $S = T$, it extends to
  all $f,g \in \Dom(T)$ by approximation. In particular, we have $\Dom(T) \subset H^2(\Omega)$. This completes the proof.
\end{proof}%

Condition~\eqref{eq:domain} usually entails that Green's formula can be justified for the partial derivatives of first order with
an overall vanishing sum of boundary terms. Examples of domains $\Omega_j$ where this can be done for $\cH_j = H_0^1(\Omega_j)$
and $\cH_j = H^1(\Omega_j)$ are discussed in a slightly different context in a recent joint work with M.~Egidi, see
Proposition~2.10 and Remark~2.4 in~\cite{ES20}. In the present note, the above criterion is demonstrated just for the case where
each $\Omega_j$ is one dimensional:

\begin{example}\label{ex:domain}
  For each $j$, suppose that $\Omega_j \subset \RR$ and that $\cH_j$ is such that $\Dom(\Delta_j) \subset H^2(\Omega_j)$ with
  $\Delta_j u = u''$ for all $u \in \Dom(\Delta_j)$. Then,~\eqref{eq:domain} is clearly satisfied, and from
  Proposition~\ref{prop:domain} it follows that $\Dom(T) \subset H^2(\Omega)$ with
  \[
    \sum_{\abs{\alpha}=2} \frac{1}{\alpha!} \langle \partial^\alpha f ,  \partial^\alpha g \rangle_{L^2(\Omega)}
    =
    \frac{1}{2} \langle Tg, Tg \rangle_{L^2(\Omega_j)}
  \]
  for all $f,g \in \Dom(T)$. Note that $\Dom(\Delta_j) \subset H^2(\Omega_j)$ with $\Delta_j u = u''$ for all
  $u \in \Dom(\Delta_j)$ certainly holds if $\cH_j$ contains $C_c^\infty(\Omega_j)$, which covers the standard choices for
  $\cH_j$ corresponding to Dirichlet, Neumann, or -- if $\Omega_j$ is bounded -- mixed or (quasi-)periodic boundary conditions.
\end{example}

\section*{Acknowledgements}
The author is grateful to A.~Dicke and M.~Egidi for fruitful discussions and helpful remarks on an earlier version of this
manuscript. He also thanks the anonymous reviewer for useful comments that helped to improve the manuscript.

%%%%%%%%%%%%%%%%%%%%%%%%%%%%%%%%%%%%%%%%%%%%%%%%%%%%%%%%%%%%%%%%%%%%%%%%%%%%%%%%%%%%%%%%%%%%%%%%%%%%%%%%%%%%%%%%%%%%%%%%%%%%%%%%%
%%%%%%%%%%%%%%%%%%%%%%%%%%%%%%%%%%%%%%%%%%%%%%%%%%%%%%%%%%%%%%%%%%%%%%%%%%%%%%%%%%%%%%%%%%%%%%%%%%%%%%%%%%%%%%%%%%%%%%%%%%%%%%%%%
%%% Bibliography
%%%%%%%%%%%%%%%%%%%%%%%%%%%%%%%%%%%%%%%%%%%%%%%%%%%%%%%%%%%%%%%%%%%%%%%%%%%%%%%%%%%%%%%%%%%%%%%%%%%%%%%%%%%%%%%%%%%%%%%%%%%%%%%%%
%%%%%%%%%%%%%%%%%%%%%%%%%%%%%%%%%%%%%%%%%%%%%%%%%%%%%%%%%%%%%%%%%%%%%%%%%%%%%%%%%%%%%%%%%%%%%%%%%%%%%%%%%%%%%%%%%%%%%%%%%%%%%%%%%

\end{document}